\documentclass[a4paper]{article}
\usepackage[english]{babel}
\usepackage{amsmath, amssymb, amsthm}
\usepackage{latexsym}
\usepackage{graphicx}
\usepackage{subfigure}
\usepackage{framed}
\usepackage{url}

\allowdisplaybreaks

\theoremstyle{plain}
\newtheorem{theorem}{Theorem}[section]
\newtheorem{lemma}[theorem]{Lemma}
\newtheorem{corollary}[theorem]{Corollary}

\theoremstyle{definition}

\let\on=\operatorname

\newcommand{\wt}[1]{\widetilde{#1}}

\newcommand{\ol}[1]{\overline{#1}}
\newcommand{\ud}{\,\mathrm{d}}

\begin{document}

\title{The energy functional on the Virasoro-Bott group with the $L^2$-metric has no local minima}
\author{Martins Bruveris$^1$}
\addtocounter{footnote}{1}
\footnotetext{Department of Mathematics, Imperial College London. London SW7~ 2AZ, UK. Partially supported by the ``Strategic and Initiative Fund'', Imperial College London
\texttt{m.bruveris08@imperial.ac.uk}}

\date{{\today}}

\maketitle

\begin{abstract}
The geodesic equation for the right invariant $L^2$-metric (which is a weak Riemannian metric) on each 
Virasoro-Bott group is equivalent to the KdV-equation. We prove that the corresponding energy functional, when restricted to paths with fixed endpoints, has no local minima. In particular solutions of KdV don't define locally length-minimizing paths. 
\end{abstract}


{\bf 2000 Mathematics Subjects Classification:} Primary 35Q53, 58B20, 58D05, 58D15, 58E12

\section{Introduction}

It was reported in \cite{OvsienkoKhesin87, Segal1991} that  a curve in the Virasoro-Bott group is a geodesic for the right invariant $L^2$-metric if and only if its right logarithmic derivative is a solution of the Korteweg--de Vries equation. The first result of this kind was \cite{Arnold1966}, where it was shown that incompressible Euler equations are geodesic equations for the group of volume-preserving diffeomorphisms. Such results can be used to show well-posedness of these equations, as done in \cite{Ebin1970}. The same methods were used later for the Camassa-Holm equation and other equations originating from Sobolev-type norms on the diffeomorphism group in \cite{Constantin2003, Constantin2007, GayBalmaz2009}.

This interpretation of hydrodynamic equations as geodesic equations on infinite dimensional manifolds opens the way for a variety of geometric questions, that can be asked. For example, the metric on the Virasoro-Bott group, for which the geodesics coincide with solutions of the KdV equation, behaves in some ways very differently from a metric on a finite dimensional manifold. It was proved in \cite{Michor122} that the geodesic distance, induced by this metric, vanishes. This means that between any two elements in the Virasoro-Bott group there exist paths of arbitrary small length. This is in stark constrast to the finite dimensional case, where the geodesic distance never vanished. Other manifolds, for which vanishing geodesic distance has been observed, include the group of diffeomorphisms with compact support $\on{Diff}_c(N)$ of any manifold $N$, the space of immersions $\on{Imm}(M,N)$ of a compact manifold $M$ into a Riemannian manifold $N$ and the corresponding shape space $\on{Imm}(M,N) / \on{Diff}(M)$, all equipped with the corresponding $L^2$-metrics. See \cite{Michor102} for further details. Other questions which might be studied include the distribution of conjugate points along geodesics \cite{Misiolek1997}, the curvature of the space and the invertibility of the exponential map \cite{Constantin2007}.

Although we know that the geodesic distance vanishes globally, this result alone does not give us any information about the local behaviour of the energy functional. For fixed, distinct endpoints $\varphi_0, \varphi_1$ geodesics between $\varphi_0$ and $\varphi_1$ can be defined as critical points of the energy functional, restricted to the space of paths with fixed end-points. It might be that (some) geodesics are locally energy-minimizing, while the potentially strange paths of arbitrary small energy lie somewhere else. Locally is to be understood with respect to some topology on the space of paths. We will use the $C^\infty$-topology on $\mathcal S([0,T]\times \mathbb R)$. In this paper we show that this is not the case. We prove that the energy functionals for the Virasoro-Bott group and for the diffeomorphism group of $\mathbb R$ have no local minima. As a corollary we get the same result for the length functional. Given any curve there is a close by curve with the same endpoints that uses less energy. Hence no geodesic is locally energy or length minimizing.

\section{The Virasoro-Bott Groups}

Let $\operatorname{Diff}_{\mathcal S}(\mathbb R)$ be the group of diffeomorphisms of $\mathbb R$, which fall rapidly to the identity. It consists of elements of the form $\operatorname{Id} + f$, where $f \in \mathcal S(\mathbb R)$ is the space of rapidly falling functions. The topology on $\mathcal S(\mathbb R)$ is induced by the family of seminorms
\[ \lVert f \rVert_{\mathcal S^{k,m}} = \sum_{i\leq m} \lVert (1+\lvert x \rvert^2)^k \partial_x^i f \rVert_\infty \]
for $k,m \in \mathbb N$. Informally the space $\mathcal S(\mathbb R)$ consists of smooth functions, which decrease to 0 faster than any rational function. This is a regular Lie group, see \cite[6.4]{Michor109} or \cite{MichorG} for more details.
For $\varphi\in\on{Diff}_{\mathcal{S}}(\mathbb R)$ let $\varphi':\mathbb R\to \mathbb R^+$ be the mapping 
given by $T_x\varphi\cdot\partial_x=\varphi'(x)\partial_x$. 
Then  
\begin{gather*} 
c:\on{Diff}_{\mathcal{S}}(\mathbb R)\times \on{Diff}_{\mathcal{S}}(\mathbb R)\to \mathbb R
\\
c(\varphi,\psi):=\frac12\int\log(\varphi\circ\psi)'\ud\log\psi'  
     = \frac12\int\log(\varphi'\circ\psi)\ud\log\psi' 
\end{gather*}satisfies $c(\varphi,\varphi^{-1})=0$, $c(\on{Id},\psi)=0$, $c(\varphi,\on{Id})=0$ and is a 
smooth group cocycle, called the Bott cocycle: 
\begin{displaymath}
c(\varphi_2,\varphi_3) - c(\varphi_1\circ\varphi_2,\varphi_3)+c(\varphi_1,\varphi_2\circ\varphi_3) -
c(\varphi_1,\varphi_2) =0.
\end{displaymath}

The corresponding central extension group
$\mathbb R\times_c\on{Diff}_{\mathcal S}(\mathbb R)$, called the Virasoro-Bott group, is
a trivial $\mathbb R$-bundle  
$\mathbb R\times\on{Diff}_{\mathcal S}(\mathbb R)$ that becomes a regular Lie  
group relative to the operations  
\begin{displaymath}
\binom{\varphi}{\alpha}\binom{\psi}{\beta}  
     =\binom{\varphi\circ\psi}{\alpha+\beta+c(\varphi,\psi)},\quad 
\binom{\varphi}{\alpha}^{-1}=\binom{\varphi^{-1}}{-\alpha}\quad 
\varphi, \psi \in \on{Diff}_{\mathcal S}(\mathbb R),\; \alpha,\beta \in \mathbb R . 
\end{displaymath}
Other versions of the Virasoro-Bott group are the following:
$\mathbb R\times_c \on{Diff}_c(\mathbb R)$ where $\on{Diff}_c(\mathbb R)$ is the group of all 
diffeomorphisms with compact support, or the periodic case $\mathbb R\times_c\on{Diff}^+(S^1)$.
One can also apply  the homomorphism $\exp(ia)$ to the center and replace it by $S^1$. 
To be specific we shall treat the most difficult case $\on{Diff}_{\mathcal{S}}(\mathbb R)$ 
in this paper. All other 
cases require only obvious minor changes in the proofs.

A short introduction of the Virasoro-Bott group can be found in \cite{Michor122}. For a detailed treatment one should consult \cite{Khesin09}, \cite{Michor109} or \cite{Michor69}.

\section{Local Minima for the Energy}

The main result of the paper for the diffeomorphism group is the following theorem.

\begin{theorem}
Let $\varphi(t,x)$ with $t \in [0, T]$ be a path in $\operatorname{Diff}_{\mathcal S}(\mathbb R)$. Let $U$ be a neighbourhood of $\varphi$ in the space $\mathcal S([0,T] \times \mathbb R)$. Then there exists a path $\psi \in U$ with the same endpoints as $\varphi$ and
\[ E(\psi) < E(\varphi), \]
where $E(.)$ is the energy of a path w.r.t. the right-invariant $L^2$-metric.
\end{theorem}

The same statement also holds for the Virasoro-Bott group.

\begin{theorem}
Let $(\varphi(t,x), \alpha(t))$ with $t \in [0, T]$ be a path in the Virasoro-Bott group $\mathbb R \times_c \operatorname{Diff}_{\mathcal S}(\mathbb R)$. Let $U$ be a neighbourhood of $(\varphi, \alpha)$ in the space $\mathcal S([0,T] \times \mathbb R) \times C^\infty([0,T])$. Then there exists a path $(\psi, \beta) \in U$ with the same endpoints as $(\varphi, \alpha)$ and
\[ E(\psi, \beta) < E(\varphi, \alpha), \]
where $E(.)$ is the energy of a path w.r.t. the right-invariant $L^2$-metric.
\end{theorem}

An immediate corollary is the following.

\begin{corollary}
The energy functionals on the diffeomorphism and Virasoro-Bott groups, even when restricted to paths with fixed endpoints, have no local minima. All stationary points are therefore saddle-points.
\end{corollary}

To proof the theorems, recall that the topology on $\mathcal S([0,T] \times \mathbb R)$ is defined using the family of seminorms
\[ \lVert \varphi \rVert_{\mathcal S^{k,m,n}} = \sum_{\substack{i\leq m \\ j \leq n}} \lVert (1+\lvert x \rvert^2)^k \partial_x^i \partial_t^j \varphi \rVert_\infty \]
with $k,m,n \in \mathbb N$ and that therefore every neighbourhood $U$ of $\varphi$ will contain an $\epsilon$-neighbourhood of one these seminorms. Therefore it is sufficient to show that there exist paths $\psi$ with the same endpoints and
\[ \lVert \psi - \varphi \rVert_{\mathcal S^{k,m,n}} < \epsilon, \]
which have less energy.

This is done in the following lemmas. We first prove it for the diffeomorphism group, where the energy is defined as
\[ E(\varphi) = \iint \varphi_t(t,x)^2 \varphi_x(t,x) \ud x \ud t, \]
In a second step we consider the Virasoro-Bott group. The energy there is defined by
\[
E(\varphi, \alpha) = \iint \varphi_t^2 \varphi_x \ud x \ud t + \int_0^T \left(\alpha_t - \int_{\mathbb R} \frac{\varphi_{tx}\varphi_{xx}}{\varphi_x^2} \ud x \right)^2 \ud t.
\]
When deforming the path there, one has to be careful that the extension part $b(t)$ of the new path $(\psi(t,x), \beta(t))$ hits the right endpoint.

\begin{lemma}
\label{lmm:en_bound}
Let $\varphi(t,x)$ with $t \in [0, T]$ be a path in $\operatorname{Diff}_{\mathcal S}(\mathbb R)$. Given $k,m,n \in \mathbb N$ there exists for small $\epsilon > 0$ and $a \in \mathbb N_{>0}$ a family of paths $\psi^\epsilon(t,x)$ with the same endpoints, close to the original path
\[ \lVert \varphi - \psi^\epsilon\rVert_{\mathcal S^{k,m,n}} < \epsilon^a \]
which use less energy. Furthermore there exists a constant $C > 0$ such that we have a lower bound for the energy saved
\[ E(\varphi) - E(\psi^\epsilon) \geq C \epsilon^{m+a}. \]
\end{lemma}

The power $a$ will be needed when we extend the argument to the Virasoro-Bott group. For now we can simply use $a=1$.

\begin{proof}
Let $r(t,x): [0,T] \times \mathbb R \to [0,T]$ be a smooth function with the property that for each $x \in \mathbb R$ it is a reparametrisation of the time interval. Define the new path via
\[ \psi(t,x) = \varphi(r(t,x),x) \]

{\bf Claim.} If $\lVert r(t,x) - t \rVert_{C^{m,n}}  < \epsilon$ then $\lVert \psi - \varphi \rVert_{S^{k,m,n}} \leq D \epsilon$ with a constant that only depends on $\lVert \varphi \rVert_{S^{k,m,n+m+1}}$.

We postpone the proof of the claim to the end of the proof. Now we want to compute $E(\psi)$.
\begin{align*}
E(\psi) &= \int_0^T \int_{\mathbb R} \psi_t^2 \psi_x \ud x \ud t \\
&= \int_0^T \int_{\mathbb R} r_t(t,x)^2 \varphi_t(r(t,x),x)^2 (r_x(t,x) \varphi_t(r(t,x),x) + \varphi_x(r(t,x),x)) \ud x \ud t \\
&= \int_0^T \int_{\mathbb R} r_t(r^{-1}(t,x),x) \varphi_t(t,x)^2 (r_x(r^{-1}(t,x),x) \varphi_t(t,x) + \varphi_x(t,x)) \ud x \ud t
\end{align*}
We can express the derivatives of $r$ using the derivatives of $r^{-1}$ via the following rules
\begin{align*}
r(r^{-1}(t,x),x) &= t \\
r_t(r^{-1}(t,x),x) &= \frac{1}{(r^{-1})_t(t,x)} \\
r_x(r^{-1}(t,x),x) &= -(r^{-1})_x(t,x) r_t(r^{-1}(t,x),x) \\
&= -\frac{(r^{-1})_x(t,x)}{(r^{-1})_t(t,x)}
\end{align*}
to obtain
\begin{align*}
L(\psi) &= \iint -(r^{-1})_x(t,x) \frac{\varphi_t(t,x)^3}{(r^{-1})_t(t,x)^2} + \frac{\varphi_t^2(t,x) \varphi_x(t,x)}{(r^{-1})_t(t,x)} \ud x \ud t \\
&= -\iint (r^{-1})_x(t,x) \frac{\varphi_t(t,x)^3}{(r^{-1})_t(t,x)^2} \ud x \ud t + \\
 &\phantom{=} + \iint \frac{1-(r^{-1})_t(t,x)}{(r^{-1})_t(t,x)} \varphi^2_t(t,x) \varphi_x(t,x) \ud x \ud t + E(\varphi)
\end{align*}
Now we have to choose $r^{-1}$ in such a way that the sum of the integrals is always negative. Let's choose a point $(t_0,x_0)$ and $\delta > 0$ such that 
\[ \varphi_t(t,x) > 0,\ \text{for}\ (t,x) \in [t_0-\delta, t_0+\delta] \times [x_0-\delta,x_0+\delta]. \]
Define
\[ r^{-1}(t,x) = t + \epsilon^{m+a} f(t) g(\frac{x-x_0}{\epsilon}). \]
We require that $f \equiv 0$ for $t \notin [t_0-\delta,t_0+\delta]$ and $f \geq 0$. For $g$ we require that $g$ be constant for $x \notin [0, 1]$ and $g' \geq 0$. The proof also works for $\varphi_t(t, x) < 0$, in this case we would require $g' \leq 0$. 

First we check that $\lVert r(t,x) - t \rVert_{C^{m,n}}  \leq \wt D \epsilon^a$ for some constant $\wt D > 0$. We see that
\[ \lVert r^{-1}(t,x) - t \rVert_{C^{m,n}} \leq \epsilon^{m+a} \lVert g(\epsilon^{-1}(x-x_0)) \rVert_{C^m} \lVert f \rVert_{C^n} \leq \epsilon^a \ol D \]
with $\ol D$ depending only on $g$ and $f$. Then we use the rules for differentiating the inverse function to see we can get a similar estimate for $r(t,x)$ with some other constant $\wt D$. Because $f$ vanishes outside a small neighbourhood of $t_0$ we have $r(0,x) = 0$ and $r(T,x) = T$ and thus the new path $\psi$ has the same endpoints as $\varphi$. 

Finally we estimate the energy gained
\begin{align*}
 -\iint &(r^{-1})_x(t,x) \frac{\varphi_t(t,x)^3}{(r^{-1})_t(t,x)^2} \ud x \ud t = \\
=& -\epsilon^{m+a-1} \int_{t_0-\delta}^{t_0+\delta} \int_{x_0}^{x_0+\epsilon} \frac{f(t) g'(\epsilon^{-1}(x-x_0))}{(1+\epsilon^{m+a}f'(t)g(\epsilon^{-1}(x-x_0)))^2} \varphi_t(t,x)^3 \ud x \ud t = \\
=& -\epsilon^{m+a} \int_{t_0-\delta}^{t_0+\delta}  \int_0^1 \frac{f(t) g'(y)}{1+\epsilon^{m+a}f'(t)g(y)} \varphi_t(t,x_0 + \epsilon y)^3 \ud y \ud t\enspace,
\end{align*}
where we used the substitution $\epsilon y = x-x_0$ in the last step. We know that $\varphi_t(t, x_0+\epsilon y) > 0$ stays away from 0 on the domain of integration by our choice of $(t_0, x_0)$ and we can approximate the denominator by
\[ \frac{1}{1+\epsilon^{m+a}f'(t)g(y)} = 1 + o(\epsilon^{m+a}), \]
which means that the whole integral has a negative part of order $\epsilon^{m+a}$ plus some terms of order at least $\epsilon^{2m+2a}$,
\[ -\iint (r^{-1})_x(t,x) \frac{\varphi_t(t,x)^3}{(r^{-1})_t(t,x)^2} \ud x \ud t < -\epsilon^{m+a} X + \epsilon^{2m+2a} (\ldots) + \ldots\]
with $X>0$. The other integral
\begin{align*}
 \iint &\frac{1-(r^{-1})_t(t,x)}{(r^{-1})_t(t,x)} \varphi^2_t(t,x) \varphi_x(t,x) \ud x \ud t = \\
=& \epsilon^{m+a} \int_{t_0-\delta}^{t_0+\delta} \int_{x_0}^{x_0+\epsilon} \frac{f'(t) g(\epsilon^{-1}(x-x_0))}{1+\epsilon^{m+a}f'(t)g(\epsilon^{-1}(x-x_0))} \varphi_t(t,x)^2 \varphi_x(t,x) \ud x \ud t = \\
=& \epsilon^{m+a+1} \int_{t_0-\delta}^{t_0+\delta} \int_{0}^{1} \frac{f'(t) g(y)}{1+\epsilon^{m+a}f'(t)g(y)} \varphi_t(t,x_0+\epsilon y)^2 \varphi_x(t,x_0+\epsilon y) \ud x \ud t
\end{align*}
is of order at least $\epsilon^{m+a+1}$. Therefore their difference has a negative part of order $\epsilon^{m+a}$ and other terms of higher order. Thus we have for small $\epsilon$ the result
\begin{align*}
E&(\varphi) - E(\psi) = \\
 =&\iint (r^{-1})_x(t,x) \frac{\varphi_t(t,x)^3}{(r^{-1})_t(t,x)^2} \ud x \ud t
 -\iint \frac{1-(r^{-1})_t(t,x)}{(r^{-1})_t(t,x)} \varphi^2_t(t,x) \varphi_x(t,x) \ud x \ud t \\
\geq& \epsilon^{m+a} X + \epsilon^{m+a+1} (\ldots).
\end{align*}
This completes the proof

{\bf Proof of claim.} The claim is essentially follows by applying Fa\`a di Bruno's formula, a higher order version of the chain rule. For $i \leq m$ and $j \leq n$ we have to estimate
\[ (1 + \lvert x \rvert^2)^k \partial_x^i \partial_t^j (\psi -  \varphi)(t,x) = (1 + \lvert x \rvert^2)^k (\partial_x^i \partial_t^j (\varphi(r(t,x),x)) - \partial_x^i \partial_t^j \varphi(t,x)) \]
First, since $\varphi(r(t,x),x)$ has two $x$-dependences, we split them
\[ \partial_x^i \partial_t^j = \partial_t^j \left( (\partial_x^i \varphi)(r(t,x),x) + \sum_{l=1}^i \binom{i}{l} \partial_x^l(\varphi \circ r) (t,x) (\partial_x^{i-l}\varphi)(r(t,x),x) \right)\]
Here we denote by $\partial_x(\varphi \circ r) (t,x)$ the differentiation of $\varphi(r(t,x),x)$ with respect to the first $x$. Each term in the sum will contain some $x$-derivative of $r$, which means that we can estimate
\[ (1+\lvert x \rvert^2)^k \partial_t^j \sum_{l=1}^i \ldots \leq D_1 \lVert \partial_x r \rVert_{C^{m-1,n}} \]
with the constant depending on $\lVert \varphi \rVert_{\mathcal S^{k, m, n+m}}$. Next we apply Fa\`a di Bruno's formula for the $t$-derivatives
\begin{align*}
 \partial_t^j &\left( (\partial_x^i \varphi )(r(t,x),x) \right) = (\partial_t^j \partial_x^i \varphi)(r(t,x),x) r_t(t,x)^i + \\
&+ j! \sum_{j>a>0} (\partial_t^a \partial_x^i \varphi)(r(t,x),x) \sum_{\substack{\alpha_1 +\ldots \alpha_a=j \\ \alpha_l > 0}} \prod_{l=1}^a \frac{\partial_t^{\alpha_l} r(t,x)}{\alpha_l!}
\end{align*}
Since in the sum $a < j$ and each $\alpha_j > 0$ in each decomposition $\alpha_1 +\ldots \alpha_a=j$ we will have at least one $\alpha_l \geq 2$. Therefore we can estimate
\[ (1+ \lvert x \rvert^2)^k j! \sum_{j > a > 0} \ldots \leq D_2 \lVert r_t \rVert_{C^{0,n-1}}\]
with $D_2$ depending on $\lVert \varphi \rVert_{\mathcal S^{k, m, n-1}}$. Next we use Taylor expansion for the remaining term
\[ (\partial_t^j \partial_x^i \varphi) (r(t,x),x) = \partial_t^j \partial_x^i \varphi(t,x) + (r(t,x)-t) \partial_t^{j+1} \partial_x^j \varphi(\xi, x) \]
with $\xi$ lying between $t$ and $r(t,x)$. Finally we note that by assumption $\lvert r_t(t,x)-1 \rvert < \epsilon$ and thus by putting it all the estimates together we obtain
\[ \lvert (1 + \lvert x \rvert^2)^k (\partial_x^i \partial_t^j (\varphi(r(t,x),x)) - \partial_x^i \partial_t^j \varphi(t,x)) \rvert \leq D \epsilon \]
for a contant $D$, depending on $\lVert \varphi \rVert_{\mathcal S^{k, m, n+m}}$ or $\lVert \varphi \rVert_{\mathcal S^{k, m, n+1}}$, in case $m=0$. This completes the proof of the claim.
\end{proof}

Now we turn our attention to the Virasoro-Bott group.

\begin{lemma}
\label{lmm:vir_en}
Let $(\varphi(t,x), \alpha(t))$ be a path in the Virasoro group $\mathbb R \times_c \operatorname{Diff}_{\mathcal S}(\mathbb R)$. Given $\delta > 0$ and $k, m, n \in \mathbb N$ with $k\geq 1$, $m\geq 2$ and $n \geq 1$, there exists a path $(\psi(t, x), \beta(t))$ close to the original path
\[ \lVert \psi - \varphi \rVert_{\mathcal S^{k, m, n}} + \lVert \alpha - \beta \rVert_{C^n} < \delta \]
with the same endpoints 
\[ \psi(0,.) = \varphi(0,.), \quad \psi(T,.) = \varphi(T,.), \quad \beta(0) = \alpha(0), \quad \beta(T) = \alpha(T) \]
and less energy
\[ E(\psi, \beta) < E(\varphi, \alpha). \]
\end{lemma}

\begin{proof}
We will first sketch the idea for the proof before going into the details. The energy for the Virasoro group is given by
\[ E(\varphi, \alpha) = \iint \varphi_t^2 \varphi_x \ud x \ud t + \int_0^T \left(\alpha_t - \int_{\mathbb R} \frac{\varphi_{tx}\varphi_{xx}}{\varphi_x^2} \ud x \right)^2 \ud t. \]
It consists of the energy of the diffeomoprhism part and a second term, that measures, how much the path deviates from a horizontal one. We will first apply lemma \ref{lmm:en_bound} to obtain a path $\wt \varphi(t,x)$ such that the energy in the diffeomorphism group is smaller, $E(\wt \varphi) < E(\varphi)$. We would like to define the extension part via
\[ \beta_t - \int_{\mathbb R} \frac{\wt \varphi_{tx}\wt \varphi_{xx}}{\wt \varphi_x^2} \ud x = \alpha_t - \int_{\mathbb R} \frac{\varphi_{tx}\varphi_{xx}}{\varphi_x^2} \ud x, \]
and of course $\beta(0) = \alpha(0)$, because this would imply that the second term of the energy remains the same. However, we need to match the endpoints of the curves as well, which means that we require
\[ \int_0^T \int_{\mathbb R} \frac{\wt \varphi_{tx}\wt \varphi_{xx}}{\wt \varphi_x^2} \ud x \ud t = \int_0^T \int_{\mathbb R} \frac{\varphi_{tx}\varphi_{xx}}{\varphi_x^2} \ud x \ud t. \]
This will in general not be the case. Therefore we perturb the path $\wt \varphi(t,x)$ a little and obtain another path $\psi(t,x)$, such that the energy of the diffeomorphism part is still less than for the original path, $E(\psi) < E(\varphi)$ and the endpoints match. Since the energy depends only on first derivatives, but the endpoint on second derivatives, we are able to move the endpoint by larger amounts, while keeping the energy close to where we started.

Now for the implementation of this plan. Let $\epsilon > 0$ be small and $\wt \varphi(t,x)$ be a path as given from lemma \ref{lmm:en_bound} (with $a=m+1$), such that $\lVert \wt \varphi - \varphi \rVert_{\mathcal S^{k,m,n}} < \epsilon^{m+1}$ and 
\[ E(\varphi) - E(\wt \varphi) \geq C \epsilon^{2m+1} \]
 for some constant $C > 0$, depending only on $\varphi$. We now define the perturbed path via
\[ \psi(t,x) = \wt \varphi(t,x) + \epsilon^{2m+\frac 3 2} \lambda f(t) g\left(\epsilon^{-2}(x-x_0)\right) \]
with some functions $f(t)$, $g(x)$, a point $x_0$ and a constant $\lambda$, which we can choose. We require $g \in \mathcal S(\mathbb R)$ to be rapidly vanishing and $f(t)$ to vanish for $t=0,T$, such that $\psi$ has the same endpoints as $\wt \varphi$. It is easy to see that $\lVert \psi - \wt \varphi \rVert_{\mathcal S^{k,m,n}} = O(\epsilon^{\frac 3 2})$. As discussed above, we modify the $\mathbb R$-component of the curve as follows
\begin{align*}
\partial_t \beta(t) &= \partial_t \alpha(t) + \int \frac{\psi_{tx} \psi_{xx}}{\psi_x^2} \ud x - \int \frac{\varphi_{tx} \varphi_{xx}}{\varphi_x^2} \ud x \\
\beta(0) &= \alpha(0).
\end{align*}

{\bf Claim A.} We have $\lVert \beta - \alpha \rVert_{C^n} \leq D \lVert \psi - \varphi \rVert_{\mathcal S^{1,2,n}}$ with a constant $D$ depending only on $\varphi$.

This claim, whose proof we postpone until later, ensures that the new path is close enough to the old one. Next we estimate the energy of the path

\begin{align*}
E_{\operatorname{Vir}}(\psi, \beta) &= E_{\operatorname{Diff}}(\psi) + \int_0^T \left( \partial_t \beta(t) - \int \frac{\psi_{tx} \psi_{xx}}{\psi_x^2} \ud x \right)^2 \ud t \\
&= E_{\operatorname{Diff}}(\psi) - E_{\operatorname{Diff}}(\varphi) + E_{\operatorname{Diff}}(\varphi) + \int_0^T \left( \partial_t \alpha(t) - \int \frac{\varphi_{tx} \varphi_{xx}}{\varphi_x^2} \ud x \right)^2 \ud t \\
&\leq \lvert E_{\operatorname{Diff}}(\psi) - E_{\operatorname{Diff}}(\wt \varphi) \rvert + E_{\operatorname{Diff}}(\wt \varphi) - E_{\operatorname{Diff}}(\varphi) + E_{\operatorname{Vir}}(\varphi, \alpha) \\
&\leq \lvert E_{\operatorname{Diff}}(\psi) - E_{\operatorname{Diff}}(\wt \varphi) \rvert - C \epsilon^{2m+1} + E_{\operatorname{Vir}}(\varphi, \alpha)
\end{align*}

{\bf Claim B.} The energy $E(\psi)$ of the perturbed path satisfies the estimate $\lvert E(\psi) - E(\wt \varphi) \rvert = O(\epsilon^{2m+\frac 3 2})$.

Therefore, for small $\epsilon$ the difference $\lvert E_{\operatorname{Diff}}(\psi) - E_{\operatorname{Diff}}(\wt \varphi) \rvert - C \epsilon^{2m+1}$ will be negative, which means that the new path uses less energy than the original,
\[ E_{\operatorname{Vir}}(\psi, b) < E_{\operatorname{Vir}}(\varphi, a). \]

Finally we need to choose $f, g, x_0$ and $\lambda$ in such a way that the endpoint $\beta(T) = \alpha(T)$ is preserved. For this we need to study the difference
\[ \int_0^T \int_{\mathbb R} \frac{\psi_{tx} \psi_{xx}}{\psi_x^2} \ud x \ud t - \int_0^T \int_{\mathbb R} \frac{\varphi_{tx} \varphi_{xx}}{\varphi_x^2} \ud x \ud t. \]
If the difference equals 0, we've accomplished our task. The neccessary second derivatives of $\psi$ are 
\begin{align*}
\psi_{tx}(t,x) &= \wt \varphi_{tx}(t,x) + \epsilon^{2m-\frac 1 2} f'(t) g'(\epsilon^{-2}(x-x_0)) \\
\psi_{xx}(t,x) &= \wt \varphi_{xx}(t,x) + \epsilon^{2m-\frac 5 2} f(t) g''(\epsilon^{-2}(x-x_0))
\end{align*}
To simplify the formulas, let us introduce the notation
\[ C(\varphi) = \iint \frac{\varphi_{tx} \varphi_{xx}}{\varphi_x^2} \ud x \ud t. \]
Then we can write the difference as
\begin{align*}
&C(\psi) - C(\varphi) = C(\psi) - C(\wt \varphi) + C(\wt \varphi) - C(\varphi) \\
&= \epsilon^{2m-\frac 5 2} \iint \frac{f(t) g''(\epsilon^{-2}(x-x_0)) \wt \varphi_{tx}(t,x)}{\wt \varphi_x(t,x)^2} \ud x\ud t + \epsilon^{2m-\frac 1 2} \iint \ldots + C(\wt \varphi) - C(\varphi)
\end{align*}
Each of the integrals contains $g(\epsilon^{-2}(x-x_0))$ or some derivative as a factor. Therefore we can perform the substitution $\epsilon^2 y = x - x_0$ to gain another factor of $\epsilon^2$. Let us define
\[ X(\epsilon) = \iint \frac{f(t) g''(y) \wt \varphi_{tx}(t,\epsilon^2 y+ x_0)}{\wt \varphi_x(t,\epsilon^2 y + x_0)^2} \ud y \ud t.\]
Choose $x_0$ such that $\varphi_{tx}(t, x_0) \neq 0$ and $g$ and $f$ such that $X(0) \neq 0$. Note that in this case $X(\epsilon)$ is bounded away from 0 for small $\epsilon$. Then
\[ C(\psi) - C(\varphi) = \epsilon^{2m-\frac 1 2} \lambda X(\epsilon) + \epsilon^{2m+\frac 3 2} \iint \ldots + C(\wt \varphi) - C(\varphi). \]
The difference $C(\wt \varphi) - C(\varphi)$ is of order $O(\epsilon^{2m})$, since $\lVert \wt \varphi - \varphi \rVert_{\mathcal S^{k, 2, 1}} = O(\epsilon^{2m-1})$ and therefore can be written as 
\[ C(\wt \varphi) - C(\varphi) = \epsilon^{2m} Y(\epsilon) \]
with $Y(\epsilon)$ bounded near $0$. With everything put together we get
\[ C(\psi) - C(\varphi) = \epsilon^{2m-\frac 1 2} (\lambda X(\epsilon) + \epsilon^{\frac 1 2}  Y(\epsilon) + \epsilon (\cdots) ) \]
This shows that by suitably choosing $\lambda$ we can achieve our goal of $C(\psi) - C(\varphi) = 0$. All that remains is to verify the claims made above and the proof will be complete.

{\bf Proof of Claim A.} For $i \leq n$ we have to estimate
\[ \partial_t^i \beta - \partial_t^i \alpha = \partial_t^{i-1} \int \frac{\psi_{tx} \psi_{xx}}{\psi_x^2} \ud x - \partial_t^{i-1} \int \frac{\varphi_{tx} \varphi_{xx}}{\varphi_x^2} \ud x. \]
The denominator doesn't present difficulties, since $\psi$ and $\varphi$ are diffeomorphisms that decay to the identity and hence $\psi_x$ and $\varphi_x$ are bounded away from 0. It is shown in \cite[6.4]{Michor109} that multiplication of rapidly decaying functions is a continuous bilinear map and from
\[ \left \lvert \int \psi \ud x \right \rvert \leq \int \frac{\ud x}{1 + x^2} \lVert (1+x^2) \psi \rVert_\infty \]
we see that integration is bounded as well. Hence we get the required estimate
\[ \lVert \beta - \alpha \rVert_{C^n} \leq D \lVert \psi - \varphi \rVert_{\mathcal S^{1,2,n}}. \]

{\bf Proof of Claim B.} The derivatives of $\psi$ are
\begin{align*}
\psi_t(t,x) &= \wt \varphi_t(t,x) + \epsilon^{2m+\frac 3 2} f'(t) g(\epsilon^{-2}(x-x_0)) \\
\psi_x(t,x) &= \wt \varphi_x(t,x) + \epsilon^{2m-\frac 1 2} f(t) g'(\epsilon^{-2}(x-x_0))
\end{align*}
With this we estimate the energy
\begin{align*}
E(\psi) &= \iint \psi_t^2 \psi_x \ud x \ud t \\
&= \iint \wt \varphi_t^2 \wt \varphi_x + \epsilon^{2m-\frac 1 2} f(t) g'( \epsilon^{-2}(x-x_0)) \psi_t^2(t,x) + \epsilon^{2m+\frac 3 2} ( \ldots ) \ud x \ud t
\end{align*}
Now we substitute in the middle term $\epsilon^2 y = x - x_0$ to gain another factor of $\epsilon^2$ and so
\begin{align*}
E(\psi) &= E(\wt \varphi) + \iint \epsilon^{2m + \frac 3 2} f(t) g'(y) \psi_t^2(t, \wt \epsilon y + x_0) \ud x \ud t + O(\epsilon^{2m+\frac 3 2})
\end{align*}
we see that
\[ \lvert E(\psi) - E(\wt \varphi) \rvert = O(\epsilon^{2m+\frac 3 2}), \]
which proves the second claim.

This completes the proof.
\end{proof}

A simple reparametrisation argument shows that the same results hold for the length functional 
\[ L(\varphi) = \int_0^T \lvert \partial_t \varphi(t) \rvert_{\varphi(t)} \ud t \]
as well.

\begin{corollary}
Let $\varphi(t,x)$ with $t \in [0, T]$ be a path in $\operatorname{Diff}_{\mathcal S}(\mathbb R)$. Given $k,m,n \in \mathbb N$ there exists for small $\epsilon > 0$ a path $\psi(t,x)$ with the same endpoints, close to the original path
\[ \lVert \varphi - \psi^\epsilon\rVert_{\mathcal S^{k,m,n}} < \epsilon \]
with smaller length
\[ L(\psi) < L(\varphi). \]
The same result holds for the Virasoro-Bott group $\mathbb R \times_c \operatorname{Diff}_{\mathcal S}(\mathbb R)$.
\end{corollary}

\begin{proof}
Given a path $\varphi(t)$, let $\wt \varphi(t) = \varphi \circ f(t)$ be a reparametrisation with constant speed. For $\wt \varphi(t)$ we find a path $\wt \psi(t)$ close to it, with less energy $E(\wt \psi) < E(\wt \varphi)$. Then let $g(t)$ be a reparametrisation, such that $\wt \psi \circ g(t)$ has constant speed and finally define $\psi(t) = \wt \psi \circ g \circ f^{-1}(t)$.

Estimating the length is simple
\begin{align*}
L(\psi)^2 &= L(\wt \psi \circ g \circ f^{-1})^2 = L(\wt \psi \circ g)^2 \\
&= \frac{1}{T} \int_0^T \lvert \partial_t( \wt \psi\circ g)(t) \rvert_{\wt \psi \circ g(t)}^2 \ud t = \frac{1}{T} E(\wt \psi \circ g) \\
&\leq \frac {1}{T} E(\wt \psi) < \frac{1}{T} E(\varphi \circ f) = \frac{1}{T} \int_0^T \lvert \partial_t(\varphi \circ f)(t) \rvert_{\varphi \circ f(t)}^2 \ud t \\
&= L(\varphi \circ f)^2 = L(\varphi)^2
\end{align*}

Finally one has to show that $\wt \psi$ being $\epsilon$-close to $\wt \varphi$ also implies that $\psi = \wt \psi \circ g \circ f^{-1}$ is $\epsilon$-close to $\varphi = \wt \varphi \circ f^{-1}$. The concatenation with $f^{-1}$ doesn't present difficulties, since we can use the chain rule to express derivatives of $\varphi \circ f^{-1}$ using derivatives of $\varphi$ and $f$ only depends on the initial path $\varphi$.

We defined $g$ to be the reparametrisation such that $\wt \psi \circ g$ has constant speed and $\wt \psi$ is $\epsilon$-close to the path $\wt \varphi$, which already has constant speed. As long as we assume that $n\geq 1$, which means that the $t$-derivatives of paths have to $\epsilon$-close to those of the original path, it follows that $g$ has to be close to the identity reparametrisation.
\end{proof}


\begin{thebibliography}{10}

\bibitem{Arnold1966}
V.~I. Arnold.
\newblock Sur la g\'eometrie diff\'erentielle des groupes de {L}ie de dimension
  infinie et ses applications \`a l'hydrodynamique des fluides parfaits.
\newblock {\em Ann. Inst. Fourier}, 16:319--361, 1966.

\bibitem{Michor122}
Martin Bauer, Martins Bruveris, Philipp Harms, and Peter Michor.
\newblock Vanishing geodesic distance for the {R}iemannian metric with geodesic
  equation the {KdV}-equation.
\newblock {\em Annals of Global Analysis and Geometry}, 41:461--472, 2012.

\bibitem{Constantin2007}
A.~Constantin, T.~Kappeler, B.~Kolev, and P.~Topalov.
\newblock On geodesic exponential maps of the {V}irasoro group.
\newblock {\em Ann. Global Anal. Geom.}, 31(2):155--180, 2007.

\bibitem{Constantin2003}
Adrian Constantin and Boris Kolev.
\newblock Geodesic flow on the diffeomorphism group of the circle.
\newblock {\em Comment. Math. Helv.}, 78(4):787--804, 2003.

\bibitem{Ebin1970}
David~G. Ebin and Jerrold Marsden.
\newblock Groups of diffeomorphisms and the motion of an incompressible fluid.
\newblock {\em Ann. of Math. (2)}, 92:102--163, 1970.

\bibitem{GayBalmaz2009}
Fran{\c{c}}ois Gay-Balmaz.
\newblock Well-posedness of higher dimensional {C}amassa-{H}olm equations.
\newblock {\em Bull. Transilv. Univ. Bra\c sov Ser. III}, 2(51):55--58, 2009.

\bibitem{Khesin09}
Boris Khesin and Robert Wendt.
\newblock {\em The Geometry of Infinite-Dimensional Groups}.
\newblock Springer, 2009.

\bibitem{MichorG}
Andreas Kriegl and Peter~W. Michor.
\newblock {\em The convenient setting of global analysis}, volume~53 of {\em
  Mathematical Surveys and Monographs}.
\newblock American Mathematical Society, Providence, RI, 1997.

\bibitem{Michor109}
Peter~W. Michor.
\newblock Some geometric evolution equations arising as geodesic equations on
  groups of diffeomorphisms including the {H}amiltonian approach.
\newblock In {\em Phase space analysis of partial differential equations},
  volume~69 of {\em Progr. Nonlinear Differential Equations Appl.}, pages
  133--215. Birkh\"auser Boston, 2006.

\bibitem{Michor102}
Peter~W. Michor and David Mumford.
\newblock Vanishing geodesic distance on spaces of submanifolds and
  diffeomorphisms.
\newblock {\em Doc. Math.}, 10:217--245 (electronic), 2005.

\bibitem{Michor69}
P.W. Michor and T.~Ratiu.
\newblock Geometry of the {V}irasoro-{B}ott group.
\newblock {\em J. Lie Theory}, pages 293--309, 1998.

\bibitem{Misiolek1997}
Gerard Misio{\l}ek.
\newblock Conjugate points in the {B}ott-{V}irasoro group and the {K}d{V}
  equation.
\newblock {\em Proc. Amer. Math. Soc.}, 125(3):935--940, 1997.

\bibitem{OvsienkoKhesin87}
V.~Y. Ovsienko and B.~A. Khesin.
\newblock {K}orteweg--de~{V}ries superequations as an {E}uler equation.
\newblock {\em Funct. Anal. Appl.}, 21:329--331, 1987.

\bibitem{Segal1991}
Graeme Segal.
\newblock The geometry of the {K}d{V} equation.
\newblock {\em Internat. J. Modern Phys. A}, 6(16):2859--2869, 1991.
\newblock Topological methods in quantum field theory (Trieste, 1990).

\end{thebibliography}
\end{document}